 \documentclass[final,1p,times]{elsarticle}
\usepackage{amssymb}
\makeatletter
\def\ps@pprintTitle{%
 \let\@oddhead\@empty
 \let\@evenhead\@empty
 \def\@oddfoot{\centerline{\thepage}}%
 \let\@evenfoot\@oddfoot}
\makeatother
\usepackage{xcolor}

\usepackage{graphics} % for pdf, bitmapped graphics files
\usepackage{epsfig} % for postscript graphics files
\usepackage{amsmath} % assumes amsmath package installed
\usepackage{amsthm}
\usepackage{amssymb}  % assumes amsmath package installed
\usepackage{subfigure}

\setcitestyle{authoryear,open={},close={}}

\newtheorem{theorem}{Theorem}%[section]
\newtheorem{corollary}{Corollary}%[theorem]
\newtheorem{remark}{Remark}%[section]
\newtheorem{example}{Example}%[section]
\newtheorem{definition}{Definition}
\newtheorem{assumption}{Assumption}

\begin{document}

\begin{frontmatter}


\title{\textbf{\LARGE Event-Triggered Control for Discrete-Time Delay Systems}}
%\title{Event-Triggered Control for Discrete-Time Delay Systems}
%
\author[QU]{Kexue Zhang}\ead{kexue.zhang@queensu.ca}

\author[UC]{Elena Braverman}\ead{maelena@ucalgary.ca}
\author[UCLA]{Bahman Gharesifard}\ead{ghasifard@ucla.edu}

\address[QU]{Department of Mathematics and Statistics, Queen's University, Kingston, Ontario K7L 3N6, Canada}

\address[UC]{Department of Mathematics and Statistics, University of Calgary, Calgary, Alberta T2N 1N4, Canada}

\address[UCLA]{Department of Electrical and Computer Engineering, University of California, Los Angeles, CA 90095, USA}

\begin{abstract}
This study focuses on event-triggered control of nonlinear discrete-time systems with time delays. Based on a Lyapunov-Krasovskii type input-to-state stability result, we propose a novel event-triggered control algorithm that works as follows. The control inputs are updated only when a certain measurement error surpasses a dynamical threshold depending on both the system states and the evolution time. Sufficient conditions are established to ensure that the closed-loop system maintains {its asymptotic stability}. It is shown that the time-dependent portion in the dynamical threshold is essential to derive the lower bound of the times between two consecutive control updates. As a special case of our results, we demonstrate the performance of the designed event-triggering algorithm for a class of linear control systems with time delays. Numerical simulations are provided to demonstrate the effectiveness of our algorithm and theoretical results.
\end{abstract}


\end{frontmatter}

\section{Introduction}\label{Sec1}

{T}{he} mechanism of event-triggered control (METC) is to update the control signals only when a certain measurement error violates a predesigned triggering condition. The advantage of METC is to {reduce the transmission load for control updates} while preserving the desired control performance. Recent years have witnessed wide applications of METC in the field of control engineering, such as, synchronization and consensus of networked systems, distributed optimization, fault detection, and sensor schedule (see, e.g., \cite{ML-2010,ZPJ-TFL:2015,CN-EG-JC:2019} and references therein).

%\margin{Is the terminology METC used elsewhere?}

Discrete-time systems are frequently encountered in digital signal processing, digital control, optimization algorithms, and digital communications (see, e.g., \cite{KJA-BW:1984,KO:1995}). In the past few years, event-triggered control for discrete-time systems has drawn lots of attention due to the advantages of METC. Numerous event-triggering algorithms have been successfully developed for many control problems (e.g., \cite{LJ-VO:2014,WW-SR-DG-SL:2016,NST-INK-KP:2016,AE-DVD-KJK:2010,PZ-TL-ZPJ:2017,SH-DY-XY-XX-YM:2016}). When the system's evolution depends on not only the current states but also the states at some previous times, examples of which can be found in coordination of multi-vehicles and  control systems with neural network inputs, the discrete-time system falls in the category of time-delay systems (see, e.g., \cite{EF-2014}). The system augmentation method, that is, converting a discrete-time system with time delays into a higher-dimensional delay-free system, has been proved to be powerful to apply results of delay-free discrete systems to the analysis of discrete-time delay systems (see, e.g., \cite{KJA-BW:1984}). Nevertheless, the generalization of the existing results on event-triggered control from discrete delay-free systems to scenarios with delay by a direct use of the system augmentation approach increases the system dimension, and requires the memory of system states at some past times. This may render the implementation of METC on discrete-time delay systems difficult. Therefore, it is crucial to study event-triggered control for discrete-time systems with time delays independently (see, e.g., \cite{BL-DJH-CZ-ZS:2018,QL-BS-ZW-TH-JL:2019,SH-XY-YZ-EGT:2012,BL-DJH-ZS-JH:2019}).

To distinguish event-triggered control from the traditional feedback control, the time difference between two consecutive control updates should be bigger than one so that the advantage of METC on efficiency improvement can be preserved; such an event-triggering algorithm is called \emph{nontrivial} (see \cite{AE-DVD-KJK:2010}). One of the main challenges in the area of event-triggered control for discrete-time systems is to ensure the nontriviality of the proposed event-triggering conditions. Various event-triggering control schemes have been successfully designed for discrete-time systems without time delays, and verifiable conditions to guarantee the nontriviality have been derived (see, e.g., \cite{AE-DVD-KJK:2010,PZ-TL-ZPJ:2017,SH-DY-XY-XX-YM:2016}). However, the study of event-triggered control for discrete-time systems with time delays is challenging, and we are only aware of very few results reported. For example, a dynamic event-triggered control algorithm was proposed in \cite{QL-BS-ZW-TH-JL:2019} to synchronize a type of discrete-time dynamical networks with time delays. Event-triggered guaranteed cost control for discrete-time systems was studied in \cite{SH-XY-YZ-EGT:2012}. Time-varying transmission delays were considered in the designed event-triggered controllers, but the uncontrolled systems were free of time delays. Unfortunately, nontriviality of the proposed event-triggering algorithms in \cite{QL-BS-ZW-TH-JL:2019,SH-XY-YZ-EGT:2012} was not discussed. By using Lyapunov function method and Razumikhin technique, several event-triggering schemes were constructed in \cite{BL-DJH-ZS-JH:2019} to stabilize a class of nonlinear discrete-time systems with time-varying delays. {The designed event-triggering conditions are nontrivial, but require the knowledge of the exact delay bound. Hence, the results in \cite{BL-DJH-ZS-JH:2019} cannot be applied to systems with bounded time delays if the information on the delay bound is unavailable.} It can be seen that the study of event-triggered control for discrete-time delay systems is to a large extent open, and derivation of sufficient conditions to guarantee the nontriviality of the event-triggering schemes is challenging.

Motivated by the above discussion, we study event-triggered control problem of discrete-time systems with time delays. The contributions of this research are summarized as follows.

\textit{Statement of Contributions.} We propose a novel event-triggered control algorithm for discrete-time delay systems, which is motivated  by a Lyapunov-Krasovskii input-to-state stability result. The designed event-triggering scheme generates control update when the measurement error reaches a dynamic threshold depending on both the system states and the evolution time. The triggering condition includes three parameters which can be tuned to ensure asymptotic stability of the control system. Sufficient conditions on these parameters are also derived to guarantee non-existence of trivial event-time sequence. Compared with the existing results, the proposed event-triggering algorithm is easy to employ, our results are applicable to discrete-time systems with time delays, and a lower bound of the inter-event times is guaranteed to be bigger than one.

The rest of this paper is structured as follows. Section~\ref{Sec02} introduces some preliminaries and a result of input-to-state stability for discrete-time delay systems. We propose a novel event-triggering scheme and establish the main results in Section~\ref{Sec03}. As an application, we investigate a type of discrete-time linear systems in Section~\ref{Sec04}. Two examples with their numerical simulations are investigated in Section~\ref{Sec05}. In Section~\ref{Sec06}, we provide a summary and discuss  possible directions for the future research.

\section{Preliminaries}\label{Sec02}
Let $\mathbb{N}$ denote the set of positive integers, $\mathbb{Z}^+$ the set of nonnegative integers, $\mathbb{R}$ the set of real numbers, $\mathbb{R}^+$ the set of nonnegative reals, and $\mathbb{R}^n$ the $n$-dimensional real space equipped with the Euclidean norm denoted by $\|\cdot\|$. For an $n\times n$ matrix $A$, we use $\|A\|$ to represent its induced matrix norm and $A^T$ to denote its transpose. For a given constant $\tau\in\mathbb{Z}^+$, let $\mathbb{N}_{-\tau}=\{-\tau,-\tau+1,...,-1,0\}$, $\mathbb{N}_{-\tau \setminus 0}=\mathbb{N}_{-\tau}\setminus \{0\}$, $\mathcal{C}_{\tau}=\{\phi:\mathbb{N}_{-\tau}\rightarrow \mathbb{R}^n\}$, and $\mathcal{C}_{\tau\setminus 0}=\{\phi:\mathbb{N}_{-\tau\setminus 0}\rightarrow \mathbb{R}^n\}$.
%\begin{align*}
%&\mathbb{N}_{-\tau}=\{-\tau,-\tau+1,...,-1,0\}, \cr
%&\mathbb{N}_{-\tau/0}=\mathbb{N}_{-\tau}/\{0\}, \cr
%&\mathcal{C}=\{\phi:\mathbb{N}_{-\tau}\rightarrow \mathbb{R}^n\}, \cr
%&\mathcal{C}_0=\{\phi:\mathbb{N}_{-\tau/0}\rightarrow \mathbb{R}^n\}.
%\end{align*}
For a given $\phi \in \mathcal{C}_{\tau}$, we define a function $\phi_{\tau\setminus 0}$ as $\phi_{\tau\setminus 0}(s)=\phi(s)$ for $s\in \mathbb{N}_{-\tau\setminus 0}$, that is, $\phi_{\tau\setminus 0}\in\mathcal{C}_{\tau\setminus 0}$. We then define two norms on $\phi$:
\[
\|\phi\|_{\tau}= \max_{s\in\mathbb{N}_{-\tau}}\left\{\|\phi(s)\|\right\}
\ \mathrm{and} \  
\|\phi\|_{\tau\setminus 0}=\max_{s\in\mathbb{N}_{-\tau\setminus 0}}\left\{\|\phi(s)\|\right\}.
\]
For a given function $u:\mathbb{Z}^+\rightarrow \mathbb{R}^m$ and $k\in\mathbb{Z}^+\setminus \{0\}$, we define $\|u\|_{[k]}=\max_{s\in\{1,2,...,k\}}\{\|u(s)\|\}$.

% as follows
%\begin{align*}
%& \|\phi\|_{\tau}= \max_{s\in\mathbb{N}_{-\tau}}\left\{\|\phi(s)\|\right\},\cr
%& \|\phi\|_{\tau/0}=\max_{s\in\mathbb{N}_{-\tau/0}}\left\{\|\phi(s)\|\right\}.
%\end{align*}
%For a given function $u:\mathbb{Z}^+\rightarrow \mathbb{R}^m$, we define
%\[
%\|u\|_{[k]}=\max_{s\in\{1,2,...,k\}}\{\|u(s)\|\}.
%\]
Next, we recall some function classes. A continuous function $\alpha:\mathbb{R}^+\rightarrow\mathbb{R}$ is said to be of class $\mathcal{K}$ and we write $\alpha\in\mathcal{K}$, if $\alpha$ is strictly increasing and $\alpha(0)=0$. If $\alpha\in\mathcal{K}$ and also $\alpha(s)\rightarrow\infty$ as $s\rightarrow\infty$, we say that $\alpha$ is of class $\mathcal{K}_{\infty}$ and we write $\alpha\in\mathcal{K}_{\infty}$. A continuous function $\beta:\mathbb{R}^+\times\mathbb{R}^+ \rightarrow\mathbb{R}^+$ is said to be of class $\mathcal{KL}$ and we write $\beta\in \mathcal{KL}$, if the function $\beta(\cdot,t)\in\mathcal{K}$ for each fixed $t\in\mathbb{R}^+$, and the function $\beta(s,\cdot)$ is decreasing and $\beta(s,t)\rightarrow 0$ as $t\rightarrow \infty$ for each fixed $s\in \mathbb{R}^+$.

Consider the discrete-time control system with time delays:
\begin{eqnarray}\label{sys}
\left\{\begin{array}{ll}
x(k+1)=f(x_k,u(k)) \cr
x_0=\varphi
\end{array}\right. ,
\end{eqnarray}
where $x(k)\in \mathbb{R}^n$ is the state and $u(k)\in \mathbb{R}^m$ is the input, for positive integers $n$ and $m$. Given $k\in \mathbb{Z}^+$, the function $x_k:\mathbb{N}_{-\tau}\rightarrow \mathbb{R}^n$ is defined as $x_k(s)=x(k+s)$ for $s\in\mathbb{N}_{-\tau}$, and the integer $\tau\geq 0$ is the maximum involved delay. We assume that $f:\mathcal{C}_{\tau}\times \mathbb{R}^m\rightarrow \mathbb{R}^n$ satisfies $f(0,0)=0$, which implies that system~\eqref{sys} admits the trivial solution (zero solution). The function $\varphi\in\mathcal{C}_{\tau}$ is the initial function, and $k=0$ is the initial time. {The notation of $x_k$ is similar to the continuous-time case for functional differential equations (see, e.g., \cite{JKH:1977}). System~\eqref{sys} is a general type of discrete-time delay systems and includes systems without time delays, systems with single or multiple discrete delays, as well as systems with time-varying bounded delays.}

The notion of input-to-state stability, introduced in \cite{EDS:1989}, and the input-to-state stability results play a significant role in designing our event-triggered control algorithm. The definition of input-to-state stability for system~\eqref{sys} is stated as follows.

\begin{definition}[see \cite{BL-DJH:2009}]
System~\eqref{sys} is said to be input-to-state stable (ISS) if there exist functions $\beta\in\mathcal{KL}$ and $\gamma\in\mathcal{K}$ such that, for each initial function $\varphi\in\mathcal{C}_{\tau}$ and input function $u:\mathbb{Z}^+\rightarrow\mathbb{R}^m$, the corresponding solution to \eqref{sys} satisfies
$$\|x(k)\|\leq \beta\left(\|\varphi\|_{\tau},k\right)+\gamma\left(\|u\|_{[k-1]}\right)~ \mathrm{~for~all~} k\in \mathbb{N}.$$
\end{definition}

Next, we introduce an ISS result for system~\eqref{sys}.
%The following theorem is a direct result of Lemma 3.5 in \cite{ZPJ-YW:2001}.
\begin{theorem}\label{Th.ISS}
Suppose there exist $V_1:\mathbb{R}^n \rightarrow \mathbb{R}^+$, $V_2:\mathcal{C}_{\tau\setminus 0}\rightarrow \mathbb{R}^+$, functions $\alpha_1,\alpha_2,\alpha_3\in \mathcal{K}_{\infty}$, $\chi\in\mathcal{K}$, and a constant $\mu\in[0,1)$, such that, for all $\phi\in\mathcal{C}_{\tau}$,
\begin{itemize}
\item[(i)] $\alpha_1(\|\phi(0)\|)\leq V_1(\phi(0)) \leq \alpha_2(\|\phi(0)\|)$;

\item[(ii)] $0\leq V_2(\phi_{\tau\setminus 0})\leq \alpha_3(\|\phi\|_{\tau\setminus 0})$;

\item[(iii)] $V(\phi):=V_1(\phi(0))+V_2(\phi_{\tau\setminus 0})$ satisfies 
\[V({\phi^*})-V(\phi)\leq -\mu V(\phi) + \chi (\|u\|),
\]
where the function ${\phi^*}:\mathbb{N}_{-\tau}\rightarrow\mathbb{R}^n$ is defined as follows
\begin{eqnarray*}
{\phi^*}(s)=\left\{\begin{array}{ll}
f(\phi,u), & \textrm{if } s=0 \cr
\phi(s+1), & \textrm{if } s\in \mathbb{N}_{-\tau\setminus 0}
\end{array}\right. .
\end{eqnarray*}
\end{itemize}
Then, system~\eqref{sys} is ISS.
\end{theorem}

The above ISS result is based on the method of Lyapunov-Krasovskii functionals. The Lyapunov-Krasovskii candidate is partitioned into a function $V_1$ of the current state and a functional $V_2$ depending only on the states at some past times. Such decomposition has been widely used in the stability analysis of discrete-time systems with time delays (see, e.g., \cite{HG-TC:2007,XM-JL-BD-HG:2010}). Define a new class $\mathcal{K}_{\infty}$ function $\bar{\alpha}_2:=\alpha_2+\alpha_3$. It follows from conditions (i) and (ii) that
\begin{equation}\label{ISS.condition}
\alpha_1(\|\phi(0)\|)\leq V( \phi ) \leq \bar{\alpha}_2(\|\phi\|_{\tau}),
\end{equation}
for all $\phi\in\mathcal{C}_{\tau}$. Starting from a Lyapunov functional $V$ in~\eqref{ISS.condition} and condition (iii), the conclusion of Theorem~\ref{Th.ISS} can be obtained by using standard Lyapunov arguments (see, e.g.,~\cite{ZPJ-YW:2001} with detailed discussions for discrete-time systems without time delays, and \cite{RHG-ML-ART:2012} with similar discussions for a class of discrete-time delay systems). Therefore, the proof of Theorem~\ref{Th.ISS} is omitted. It is worthwhile to mention that the decomposition of $V$ into $V_1$ and $V_2$ is not necessary to guarantee the ISS property. Nevertheless, the function portion $V_1$ coupled with condition (iii) in Theorem~\ref{Th.ISS} is essential for developing our event-triggering algorithm.

\section{Event-Triggered Control Algorithm}\label{Sec03}
Consider feedback control system~\eqref{sys} with a sampled-data implementation
\begin{eqnarray}\label{ETC.sys}
\left\{\begin{array}{ll}
x(k+1)=f(x_k,u(k)) \cr
u(k)=p(x(k_i)),~k\in[k_{i},k_{i+1})\cr
x_0=\varphi
\end{array}\right. ,
\end{eqnarray}
where $u:\mathbb{Z}^+\rightarrow \mathbb{R}^m$ is a feedback control input, $p:\mathbb{R}^n\rightarrow \mathbb{R}^m$ is the feedback control law and satisfies $p(0)=0$. Thus, system~\eqref{ETC.sys} admits a trivial solution as $f(0,0)=0$. The time sequence $\{k_i\}_{i\in \mathbb{N}}\subset \mathbb{Z}^+$ is a set of discrete moments when the control signals are updated and will be determined by a certain execution rule based on the state measurement.

To introduce our execution rule, we first define the state measurement error
\[
e(k)=x(k_i)-x(k), \qquad k\in[k_{i},k_{i+1}).
\]
We have that
\[
p(x(k_i))=p(e(k)+x(k)),
\]
where $k\in[k_{i},k_{i+1})$, and the control system~\eqref{ETC.sys} can be written as
\begin{eqnarray}\label{CL.sys}
\left\{\begin{array}{ll}
x(k+1)=g(x_k,e(k)) \cr
x_0=\varphi
\end{array}\right. ,
\end{eqnarray}
with $g(x_k,e(k)):=f(x_k,p(e(k)+x(k)))$ for all $k\in\mathbb{Z}^+$.

Throughout this paper, we make the following assumption on system~\eqref{CL.sys}.
\begin{assumption}\label{Assumption}
Suppose there exist $V_1:\mathbb{R}^n \rightarrow \mathbb{R}^+$, $V_2:\mathcal{C}_{\tau\setminus 0}\rightarrow \mathbb{R}^+$, functions $\alpha_1,\alpha_2,\alpha_3\in \mathcal{K}_{\infty}$, $\chi\in\mathcal{K}$, and a constant $\mu\in[0,1)$, such that the conditions of Theorem \ref{Th.ISS} hold for system~\eqref{CL.sys} with $f$ and $u$ replaced by $g$ and $e$, respectively.
\end{assumption}

Under the above assumption, system~\eqref{CL.sys} is globally asymptotically stable without the measurement error $e$ and ISS with respect to $e$. The objective of this study is to design a feasible execution rule to determine the sequence $\{k_i\}_{i\in \mathbb{N}}$ {so that the closed-loop system~\eqref{CL.sys} preserves its global asymptotic stability.}
%\begin{definition}[Boundedness]
%System~\eqref{CL.sys} is said to be {uniformly bounded (UB)}, if for any $\delta>0$ there exists a constant $M:=M(\delta)>0$ such that
%\[
%\|\varphi\|_{\tau}\leq \delta \Rightarrow \|x(k)\|\leq M ~\textrm{ for all } k\in \mathbb{N},
%\]
%where $x(k):=x(k,0,\varphi)$ is the solution of system~\eqref{CL.sys}.
%\end{definition}
%
\begin{definition}[Attractivity]
The {trivial solution} of system~\eqref{CL.sys} is said to be globally attractive (GA), if 
\[
\lim_{k\rightarrow\infty} \|x(k)\|=0 ~\textrm{ for any } \varphi\in\mathcal{C}_{\tau},
\]
where $x(k):=x(k,0,\varphi)$ is {the solution} of~\eqref{CL.sys}.
\end{definition}

\begin{definition}[Global Asymptotic Stability]
The {trivial solution} of system~\eqref{CL.sys} is said to be globally asymptotically stable (GAS) if it is stable and globally attractive.
\end{definition}

To derive the time sequence $\{k_i\}_{i\in \mathbb{N}}$, we enforce $e$ to satisfy
\begin{equation}\label{restrict}
\chi(\|e\|) \leq \sigma \alpha_1(\|x\|) + \chi\left( a(1-b)^k \right),
\end{equation}
for some constants $\sigma\geq 0$, $a\geq 0$, and $1>b>0$ to be determined later. The updating of the control input $u$ is triggered by the following execution rule (or event)
\begin{equation}\label{event}
\chi(\|e\|) > \sigma \alpha_1(\|x\|) + \chi\left( a(1-b)^k \right).
\end{equation}
The event times are the moments when the event occurs, i.e.,
\begin{equation}\label{et.time}
k_{i+1}=\min\left\{ k> k_i \mathrel{\big|} \chi(\|e(k)\|) > \sigma \alpha_1(\|x(k)\|) + \chi\left( a(1-b)^k \right) \right\}.
\end{equation}
Execution rule~\eqref{event} works as follows. At each event time $k_i$, the input signals are updated according to the feedback control law introduced in system~\eqref{ETC.sys}, and the measurement error $e$ is set to zero. The control input $u$ remains constant in the following time steps until the error $e$ violates the requirement~\eqref{restrict} at the next event time $k_{i+1}$. Then, the control input is renewed as $u(k)=p(x(k_{i+1}))$, and the measurement error is reset to zero again. This process is repeated for every time period in between consecutive events, i.e., $[k_i,k_{i+1})$ for $k\in \mathbb{N}$. As the sequence of event times $\{k_i\}_{i\in\mathbb{N}}$ is defined in an implicit manner, it is possible for the input $u$ to be updated at every time step, that is, $k_{i+1}-k_i=1$ for $k\in\mathbb{N}$. For this scenario, the event-triggered control system~\eqref{ETC.sys} reduces to the traditional feedback control system, and the advantages of the event-triggered control mechanism vanish. Therefore, in order to preserve the efficiency of event-triggered control for discrete-time systems, it is important to secure that the lower bound of the \textit{inter-execution times} $\{k_{i+1}-k_i\}_{i\in\mathbb{N}}$ is bigger than one (i.e., $k_{i+1}-k_i \geq 2$ for all $i\in\mathbb{N}$, that is, the control input $u$ is updated at most every other time step). We call such a sequence of event times {\textit{strongly nontrivial}. If there exists at least one $i\in\mathbb{N}$ so that $k_{i+1}-k_i \geq 2$, then the sequence $\{k_i\}_{i\in\mathbb{N}}$ is called \textit{weakly nontrivial}. It can be observed that an event-triggering algorithm with weakly nontrivial sequence of event times may not be able to secure a significant advantage of event-triggered control over the conventional feedback control in reducing control updates. Because a weakly nontrivial sequence may only allow $k_{i+1}-k_i \geq 2$ for some finite numbers of inter-execution times. Therefore, we focus on the existence of strongly nontrivial sequence of event times according to the proposed event-triggering algorithm.}

{
\begin{remark}\label{remark0}
It is worthwhile to mention that the triggering condition~\eqref{event} is inspired by Theorem~\ref{Th.ISS}. If we enforce $\chi(\|e\|) > \sigma \alpha_1(\|x\|)$, then condition (iii) of Theorem~\ref{Th.ISS} ensures exponential convergence of the Lyapunov functional candidate. Hence, we can conclude from Theorem~\ref{Th.ISS} that the closed-loop system is asymptotically stable. However, nontrivial control updates cannot be guaranteed (see Section~\ref{Sec05} for an example). Therefore, the triggering condition needs to be revised in order to enlarge the inter-execution times. The time-dependent portion in~\eqref{event} plays an essential role to ensure the nontrivial control updates, see the proof of Theorem~\ref{Th1}.
\end{remark}
}

In what follows, we will establish several sufficient conditions to guarantee that closed-loop system~\eqref{CL.sys} with event times determined by~\eqref{et.time} is still {asymptotically stable} and assures the notriviality of the sequence of event times, simultaneously. {Our main results mainly rely on the Lipschitz conditions of the functions that appear in Assumption~\ref{Assumption}.}

%{\color{red} \bf B: If $\tau=0 $, $\mathcal{C}_{0}$ includes only the zero function, correct? Does that mean the first definition becomes a special case of the second? If so, we can remove (8) all together.}

{
\begin{definition}[Lipschitz]\label{Lipschitz} 
The function $f:\mathcal{C}_{\tau}\times\mathbb{R}^m\rightarrow\mathbb{R}^n$ is called locally Lipschitz, if for each $\psi\in \mathcal{C}_{\tau}$ and $u\in\mathbb{R}^m$ there exist positive constants $l_{1,1}$, $l_{1,2}$, $l_2$, and $R$ such that 
\begin{align}\label{Lipschitz02}
\|f(\phi,v)-f(\psi,u)\| \leq & l_{1,1} \|\phi(0)-\psi(0)\| + l_{1,2} \|\phi-\psi\|_{\tau\setminus 0} \cr
                                & + l_2 \|v-u\|
\end{align}
for all $\phi$ in the open ball of center $\psi$ and radius $R$:
\[
\mathcal{B}^{\tau}_{R}(\psi):=\{\phi\in\mathcal{C}_{\tau} \mid \|\phi-\psi\|_{\tau} < R  \}
\]
and all $v$ in the open ball of center $u$ and radius $R$:
\[
\mathcal{B}_{R}(u):= \{v\in\mathbb{R}^m \mid \|v-u\| < R  \}.
\]
The Lipschitz condition for single-variable functions can be derived from~\eqref{Lipschitz02} with the first argument of $f$ fixed.
\end{definition}
}
With Definition~\ref{Lipschitz}, we propose the second assumption.
\begin{assumption}\label{Assumption2}
The following Lipchitz conditions are satisfied.
\begin{itemize}
\item $\alpha^{-1}_1$ (i.e., the inverse of $\alpha_1$), $\chi$, and $p$ are locally Lipschitz;

\item {$(\phi,u)\mapsto f(\phi,u)$ is locally Lipschitz.}

\end{itemize}
\end{assumption}

%%DISCUSSION of the Lipschitz conditions for $f$.
To state our main results, we let $c:=\mu-\sigma$ and $\tilde{M}:=\alpha_2(\|\varphi(0)\|)+\alpha_3(\|\varphi\|_{\tau\setminus 0})$. With Assumption~\ref{Assumption2}, we denote $L$ as the Lipschitz constant of $\chi$ on the closed interval $[0,a]$, and then, for any given constant $\xi$ so that $0<\xi<b$, we define
\begin{align}\label{Mbar}
\bar{M}=\left\{\begin{array}{ll}
\frac{aL}{|c-b|},~&\textrm{if}~ b\not=c \cr
\frac{aL}{(1-c)\left(\ln({1-\xi}) -  \ln({1-c})\right)},~&\textrm{if}~ b=c<1
\end{array}\right. 
\end{align}
and $M:=\bar{M} + \tilde{M}$. We further denote by $L_1$ the Lipschitz constant of $\alpha^{-1}_1$ on the interval $[0,M]$. 
Define function $\bar{f}:\mathcal{C}_{\tau}\times \mathbb{R}^m\rightarrow \mathbb{R}^n$ as 
\[
\bar{f}(\phi,u):=\phi(0)-f(\phi,u),
\]
for $\phi\in\mathcal{C}_{\tau}$ and $u\in\mathbb{R}^m$. {The function $(\phi,x)\mapsto \bar{f}(\phi,p(x))$  is locally Lipschitz, since $p$ is locally Lipschitz. Given this, we let $\bar{L}_{1,1}$, $\bar{L}_{1,2}$, and $\bar{L}_2$ be  Lipschitz constants such that 
\begin{equation}\label{Lipschitz_fbar}
\|\bar{f}(\phi,p(x))\|\leq \bar{L}_{1,1}\|\phi(0)\| + \bar{L}_{1,2}\|\phi\|_{\tau\setminus 0}+\bar{L}_2\|x\|,
\end{equation}
where  $\phi\in\mathcal{B}^{\tau}_{R}(0)$ and $x\in\mathcal{B}_R(0)$ with radius $R=\alpha^{-1}_1(M)$. It should be noted that $\tilde{M}$, $\bar{M}$, and the above-named Lipschitz constants depend on the initial condition of system~\eqref{ETC.sys}.}

%{\color{red} \bf B: I suggest replacing the second sentence with this:}

%%{\color{magenta} Given this, we let $\bar{L}_{1,1}$, $\bar{L}_{1,2}$, and $\bar{L}_2$ be  Lipschitz constants such that 
%\begin{equation}
%\|\bar{f}(\phi,p(x))\|\leq \bar{L}_{1,1}\|\phi(0)\| + \bar{L}_{1,2}\|\phi\|_{\tau\setminus 0}+\bar{L}_2\|x\|,\nonumber
%\end{equation}
%where  $\phi\in\mathcal{B}^{\tau}_{R}(0)$ and $x\in\mathcal{B}_R(0)$ with radius $R=\alpha^{-1}_1(M)$.}

Now we are ready to state our first main result.

\begin{theorem}\label{Th1}
Suppose that both Assumptions~\ref{Assumption} and~\ref{Assumption2} hold. The event times $\{k_i\}_{i\in\mathbb{N}}$ in~\eqref{ETC.sys} are determined by~\eqref{et.time} with $\mu>\sigma\geq 0$, $a\geq 0$, and $1>b>0$. Then,
\begin{itemize}
\item[1)] if $a=0$ and $\sigma>0$, the closed-loop system~\eqref{CL.sys} is GAS;

\item[2)] {if $a>0$, the closed-loop system~\eqref{CL.sys} is also GAS. If we further assume that $b<c$, and parameters $a$, $b$, and $\sigma$ satisfy the following inequality
\begin{equation}\label{condition}
a \geq \frac{M L_1 }{1-b}\left( \bar{L}_{1,1} +\frac{\bar{L}_{1,2}}{(1-b)^{\tau}} + \bar{L}_2 \right),
\end{equation}
then the sequence of event times $\{k_i\}_{i\in\mathbb{N}}$ is strongly nontrivial, that is, $k_{i+1}-k_i \geq 2$ for all $i\in\mathbb{Z}^+$. }

\end{itemize}
\end{theorem}

%\margin{Do we have an example of $ a>0 $ which is not GAS?}

%\#\#\#\#\#\#\#\#\#\#\#\#\#\#
%%%
\begin{proof} {Consider control system~\eqref{ETC.sys}, and denote $x(t):=x(t,0,\varphi)$ its solution.}
It follows from the restriction~\eqref{restrict} and the Lipschitz condition of $\chi$ in Assumption~\ref{Assumption2} that
\begin{equation}\label{restrict1}
\chi(\|e\|) \leq \sigma \alpha_1(\|x\|) +  a L (1-b)^k ,
\end{equation}
for all $k\in\mathbb{N}$. We then conclude from {condition (iii)} of Theorem~\ref{Th.ISS} and~\eqref{restrict1} that
\begin{align*}
V(x_{k+1})-V(x_k) &\leq -\mu V(x_k) + \chi (\|e\|) \cr
               &\leq -\mu V(x_k) + \sigma \alpha_1(\|x\|) + \chi\left( a(1-b)^k \right)\cr
               &\leq -c V(x_k) +  a L (1-b)^k, 
\end{align*}
that is,
\begin{equation}\label{base}
V(x_{k+1}) \leq (1-c) V(x_k) +  a L (1-b)^k ,
\end{equation}
for all $k\in\mathbb{Z}^+$. Using~\eqref{base} for $k$ times yields
{\begin{align}\label{base1}
V(x_{k+1}) &\leq (1-c) \left( (1-c) V(x_{k-1}) +  a L (1-b)^{k-1}  \right) +  a L (1-b)^k  \cr
           &  =  (1-c)^2 V(x_{k-1}) + \sum_{j=0}^{1} a L (1-c)^j (1-b)^{k-j} \cr
           &\leq (1-c)^2 \left( (1-c) V(x_{k-2}) +  a L (1-b)^{k-2}  \right) \cr
           &\hskip3mm + \sum_{j=0}^{1} a L (1-c)^j (1-b)^{k-j} \cr
           &  =  (1-c)^3 V(x_{k-2}) + \sum_{j=0}^{2} a L (1-c)^j (1-b)^{k-j} \cr
%           &\leq \cdots~\cdots \cr
           &\leq (1-c)^{k+1} V(x_0) +\sum_{j=0}^{k} a L (1-c)^j (1-b)^{k-j}  ,
\end{align} }
for all $k\in\mathbb{Z}^+$.

If $b>c$, we derive from~\eqref{base1} that
\begin{align}\label{case1}
V(x_{k+1}) &\leq (1-c)^{k+1} \left( V(x_0) + \frac{a L}{1-c} \sum_{j=0}^{k} \left( \frac{1-b}{1-c} \right)^{j} \right) \cr
  &= (1-c)^{k+1} \left( V(x_0) + \frac{a L}{1-c} \frac{1- \left( \frac{1-b}{1-c} \right)^{k+1} }{1- \left( \frac{1-b}{1-c} \right)}  \right) \cr
  &\leq (1-c)^{k+1} \left( V(x_0) + \frac{a L}{b-c}  \right).
\end{align}

Similarly, if $b<c$, we have
\begin{align}\label{case2}
V(x_{k+1}) &\leq (1-b)^{k+1} \left( \left(\frac{1-c}{1-b} \right)^{k+1} V(x_0) + \frac{a L}{1-b} \sum_{j=0}^{k} \left( \frac{1-c}{1-b} \right)^{j} \right) \cr
  &= (1-b)^{k+1} \left( V(x_0) + \frac{a L}{1-b} \frac{1- \left( \frac{1-c}{1-b} \right)^{k+1} }{1- \left( \frac{1-c}{1-b} \right)}  \right) \cr
  &\leq (1-b)^{k+1} \left( V(x_0) + \frac{a L}{c-b}  \right).
\end{align}

If $b=c$, we conclude from~\eqref{base1} that 
\begin{align}\label{case3}
V(x_{k+1}) &\leq (1-c)^{k+1}  V(x_0) + \frac{a L}{1-c} (k+1) (1-c)^{k+1}   \cr
 &\leq (1-c)^{k+1}  V(x_0) + (1-\xi)^{k+1}  \frac{a L}{(1-c) \ln \left( \frac{1-\xi}{1-c} \right)} \cr
 &\leq (1-\xi)^{k+1} \left(  V(x_0) + \frac{a L}{(1-c) \ln \left( \frac{1-\xi}{1-c} \right)}  \right),
\end{align}
where we used the fact
\[
(k+1) \ln\left( \frac{1-\xi}{1-c} \right) \leq \left( \frac{1-\xi}{1-c} \right)^{k+1}
\]
with $0<\xi<c$. Denote
\[
\eta:=\left\{\begin{array}{ll}
\min\{b,c\},~&\textrm{if}~ b\not=c \cr
\xi,~&\textrm{if}~ b=c
\end{array}\right.,
\]
then we conclude from~\eqref{case1},~\eqref{case2}, and~\eqref{case3} that
\begin{equation}\label{Lyapunov.convergence}
V(x_{k+1}) \leq (1-\eta)^{k+1} \left( V(x_0) + \bar{M} \right) \leq M (1-\eta)^{k+1},
\end{equation}
for all $k\in\mathbb{Z}^+$. From the above inequality and conditions (i) and (ii) of Theorem~\ref{Th.ISS}, we get
\begin{equation}\label{bounded}
\|x(k)\|\leq \alpha^{-1}_1 \left(M (1-\eta)^{k}\right),
\end{equation}
for all $k\in\mathbb{Z}^+$, and
\begin{equation}\label{attactive}
\lim_{k\rightarrow\infty} \|x(k)\|\leq \lim_{k\rightarrow\infty} \alpha^{-1}_1 \left(M (1-\eta)^{k}\right) =0.
\end{equation}
Hence, the trivial solution of system~\eqref{CL.sys} is GA. {To show that system~\eqref{CL.sys} is GAS, we prove that system~\eqref{CL.sys} is stable.}

%{\color{red} \bf B: What does which refer to here?}

If $a=0$, then $\bar{M}=0$, and $M=\tilde{M}$ depends only on the initial function $\varphi$. We then can conclude from~\eqref{bounded} the stability of system~\eqref{CL.sys} {which also can be derived from \eqref{base} with $a=0$ and Theorem \ref{Th.ISS}.}

%{\color{red} \bf B: It think in below we only need to say  Lipschitz conditions on $p$; isn't $f$ implied by it?}

{
If $a>0$, we can derive from system~\eqref{ETC.sys} and~\eqref{Lipschitz_fbar} that
\begin{align}\label{xat1}
\|x(1)\|&\leq \bar{L}_{1,1}\|x(0)\| +\bar{L}_{1,2}\|x_0\|_{\tau\setminus 0} + \bar{L}_2 \|x(0)\| \cr
        &\leq \left(\bar{L}_{1,1} +\bar{L}_{1,2} + \bar{L}_2\right) \|\varphi\|_{\tau}.
\end{align}
Using a mathematical induction and~\eqref{xat1}, we conclude that
\begin{align}\label{xatk}
\|x(k)\|<\left\{\begin{array}{ll}
\|\varphi\|_{\tau},~&\textrm{if}~ \bar{L}_{1,1} +\bar{L}_{1,2} + \bar{L}_2<1 \cr
\left(\bar{L}_{1,1} +\bar{L}_{1,2} + \bar{L}_2\right)^k \|\varphi\|_{\tau},~&\textrm{otherwise}
\end{array}\right. 
\end{align}
for all $k\geq 0$. Stability of the closed-loop system follows directly from the above inequality if $\bar{L}_{1,1} +\bar{L}_{1,2} + \bar{L}_2<1$. Thus we next focus on the scenario of  $\bar{L}_{1,1} +\bar{L}_{1,2} + \bar{L}_2\geq 1$. For small positive constant $\delta$ such that $\delta<L_1M$, the following two inequalities
\[
\left(\bar{L}_{1,1} +\bar{L}_{1,2} + \bar{L}_2\right)^{k^*} \delta\leq L_1M(1-\eta)^{k^*}
\]
and
\[
\left(\bar{L}_{1,1} +\bar{L}_{1,2} + \bar{L}_2\right)^{k^*+1} \delta > L_1M(1-\eta)^{k^*+1}
\]
hold with 
\[
k^*= \left\lfloor\frac{\ln\left( \frac{\delta}{L_1M} \right)}{\ln\left( \frac{1-\eta}{\bar{L}_{1,1} +\bar{L}_{1,2} + \bar{L}_2} \right)}\right\rfloor
\]
where $\lfloor\cdot\rfloor$ is the floor function. It then can be concluded from~\eqref{xatk} and~\eqref{bounded} that for $\|\varphi\|_{\tau}< \delta$, we have
\begin{align*}
\|x(k)\|<\left\{\begin{array}{ll}
\left(\bar{L}_{1,1} +\bar{L}_{1,2} + \bar{L}_2\right)^{k^*} \delta,~&\textrm{if}~ k\leq k^* \cr
L_1M(1-\eta)^{k^*},~&\textrm{if}~ k> k^*,
\end{array}\right. 
\end{align*}
which implies that $\|x(k)\|<L_1M(1-\eta)^{k^*}$ for all $k\geq 0$. Therefore, for any $\varepsilon>0$, there exists a positive $\delta$ close enough to zero, depending on $\varepsilon$, such that $k^*$ is big enough so that $L_1M(1-\eta)^{k^*}\leq \varepsilon$, that is, $\|\varphi\|_{\tau}<\delta$ implies
\[
\|x(k)\|<L_1M(1-\eta)^{k^*}\leq \varepsilon ~\textrm{~for ~all~}~ k\geq 0.
\]
Hence, system~\eqref{CL.sys} is stable.
}

%{\color{red}\bf B: I checked this, but maybe look once more to make sure we use $\varphi $ and $ \phi $ consistently. It is fine anywhere I checked. The only point that you could consider changing is in Definition 4, because in there you use it as an arbitrary point, whereas most places it refers to the initial in~(3). You can consider $ \phi_1 $ and $\phi_2 $ in that definition, saving $\varphi $ for~(3)}

In the rest of this proof, we will show that $k_{i+1}-k_i\geq 2$ for all $k\in\mathbb{Z}^+$ provided $b<c$. We do this by a contradiction argument. Suppose that there exists some $i\in\mathbb{Z}^+$ with $k_{i+1}-k_i=1$, i.e., $k_{i+1}=k_i+1$.
{By~\eqref{ETC.sys}, we have that}
%{\color{red}\bf B: I removed one line below}
%\margin{It seems clear, but basically the definition of $ \bar{f} $ is used on the interval $ [k_i,k_{i+1}] $, correct?}
\begin{align}\label{contradiction0}
                  & \|x(k_{i+1})-x(k_i)\| \cr
                 =& \|x(k_{i}+1)-x(k_i)\| \cr
%                 =& \left\|f(x_{k_i},p(x(k_i)))-x(k_i)\right\| \cr
                 =& \left\|\bar{f}(x_{k_i},p(x(k_i)))\right\| \cr
             \leq & \bar{L}_{1,1} \|x(k_i)\| + \bar{L}_{1,2} \left\|x_{k_i}\right\|_{\tau\setminus 0} + \bar{L}_{2} \|x(k_i)\| \cr
             \leq & \left(\bar{L}_{1,1}+\bar{L}_2\right) \alpha^{-1}_1\left( M (1-\eta)^{k_i} \right)  +  \bar{L}_{1,2} \alpha^{-1}_1\left( M (1-\eta)^{k_i-\tau} \right) \cr
             \leq & M L_1 (1-\eta)^{k_i} \left( \bar{L}_{1,1} + \frac{\bar{L}_{1,2}}{(1-\eta)^{\tau}} + \bar{L}_2 \right) \cr
             \leq & a(1-b) (1-\eta)^{k_i} \cr
               =  & a(1-b)^{k_i+1} ,
\end{align}
where we used the Lipschitz condition~\eqref{Lipschitz_fbar} in the first inequality of~\eqref{contradiction0}, the Lipschitz condition of $\alpha^{-1}_1$ on the interval $[0,M]$ in the second inequality, inequality~\eqref{condition} in the third inequality, and $\eta=b$ with the condition $b<c$ in the last equality of~\eqref{contradiction0}. 
%{\color{red} The above-mentioned Lipschitz conditions are valid for all event times, which is ensured by \eqref{bounded}.}
%{\color{red}\bf B: Can we remove this last sentence? We have already said for all $k $ in our argument}

According to execution rule~\eqref{event} with the definition of $k_{i+1} $, we have
\begin{align*}
\sigma \alpha_1(\|x(k_{i+1})\|) + \chi\left( a(1-b)^{k_{i+1}} \right) &<\chi\left(\|x(k_{i+1})-x(k_{i})\|\right) \cr
     &\leq \chi\left( a(1-b)^{k_i+1} \right),
\end{align*}
{a contradiction, yielding that $k_{i+1}-k_i\geq 2$ for all $i\in\mathbb{Z}^+$.} 
\end{proof}

%To derive the lower bound of the inner-execution times $\{k_{i+1}-k_i\}_{i\in\mathbb{N}}$, we mainly assume the Lipschitz conditions on $\alpha^{-1}_1$, $p$, and $f$. By using the dynamics of system~\eqref{CL.sys}, we estimate $\|e(k_{i+1})\|$ and we have derived the following inequality:
%\begin{equation}\label{inequality}
%a(1-\eta)^{\Delta_{i+1}}<\lambda_1\left(1-(1-\eta)^{\Delta_{i+1}}\right) + \lambda_2 \Delta_{i+1}
%\end{equation}
%where $\Delta_{i+1}=k_{i+1}-k_i$ and
%\begin{align*}
%\lambda_1&= \frac{ML_1}{b}\left( \bar{L}_{1,1}+\frac{\bar{L}_{1,2}}{(1-b)^{\tau}} \right)\cr
%\lambda_2&=M L_1\bar{L}_2.
%\end{align*}
%If \eqref{condition} holds, then
%\begin{equation}\label{inequality1}
%a(1-\eta) \geq \lambda_1 \eta + \lambda_2.
%\end{equation}
%Since $\Delta_{i+1}\in\mathbb{N}$, we conclude from~\eqref{inequality} and~\eqref{inequality1} that $\Delta_{i+1}\geq 2$ for all $i\in\mathbb{N}$.

According to our event-triggering scheme, the control signals are updated when the quantity $\chi(\|e\|)$ of the measurement error goes over the dynamic threshold $\sigma \alpha_1(\|x\|) + \chi\left( a(1-b)^{k} \right)$ which depends on both the system states and the evolution time. If $\sigma>0$ and $a=0$, our execution rule~\eqref{event} becomes
\begin{equation}\label{event1}
\chi(\|e\|) > \sigma \alpha_1(\|x\|) ,
\end{equation}
which has been studied for discrete-time systems without time delays in~\cite{AE-DVD-KJK:2010}. Theorem~\ref{Th1} says that time-delay system~\eqref{CL.sys} is GAS. In this sense, we generalize the results in~\cite{AE-DVD-KJK:2010} for delay-free systems to deal with time-delay systems. Unfortunately, no effective approaches were provided in~\cite{AE-DVD-KJK:2010} to guarantee the nontrivial control updates. Actually, we will show in Example~\ref{ex2} with Fig.~\ref{fig2} that such trivial scenario of control updates indeed exists for some time-delay systems, which implies that the execution rule~\eqref{event1} may trigger the control updates too frequently for certain discrete-time systems. To overcome this problem, the intuitive idea to add the time-dependent part $\chi\big( a(1-b)^k \big)$ in our execution rule is to enlarge the time for the quantity $\chi(\|e\|)$ to evolve from zero to the time of resetting $e$ to zero again. Therefore, the updating of the control signals is most likely triggered less frequently.

The existence of such time-dependent portion in~\eqref{event} is essential to assure the lower bound of inter-execution times is bigger than one. The designable parameters $a$ and $b$ play important roles in the performance of the proposed algorithm (e.g., bound of the system trajectories, convergence speed of the Lyapunov candidate, and the number of the event times). For instance, setting $a$ large reduces the amount of events at the cost of increasing the bound of system trajectories. Setting $b$ large with $b<c$ increases the convergence speed but more events are triggered. If $\sigma=0$ in ~\eqref{event}, we derive the following execution rule
\begin{equation}\label{event2}
\|e(k)\| > a (1-b)^k,
\end{equation}
which relies only on the evolution time $k$. The advantage of the time-dependent event~\eqref{event2} is its simplicity to design and implement. Nevertheless, less control updates may be triggered by execution rule~\eqref{event}. The above discussions are further demonstrated with numerical simulations in Section~\ref{Sec05}.

%{\color{red} \bf B: I got a bit confused below. First, is $\bar{L}_2$ the same constant as we had before? If so, I rewrote one part.}

\begin{remark}\label{remark.Lipschitz}
{The Lipschitz condition on $\bar{f} $ can be replaced with 
\begin{align}\label{Lipschitz1}
\|\bar{f}(\phi,p(x))\| \leq & \bar{L}_1 \|\phi\|_{\tau} + \bar{L}_2 \|x\|
\end{align}
for all $\phi\in \mathcal{B}^{\tau}_R(0)$ and $x\in\mathcal{B}_R(0)$ where $R=\alpha^{-1}_1(M)$ and 
 %$\bar{L}_1$ and $\bar{L}_2$ are Lipschitz constants with  
 $\bar{L}_1=\bar{L}_{1,1}+\bar{L}_{1,2}$.
%and , which can be obtained from~\eqref{Lipschitz02} with $\bar{L}_1=\bar{L}_{1,1}+\bar{L}_{1,2}$.
} 
If Lipschitz condition~\eqref{Lipschitz1} holds for $\bar{f}$ instead, then 
%from~\eqref{contradiction0} 
the condition~\eqref{condition} in Theorem~\ref{Th1} can be replaced by
\begin{equation}\label{conditionL}
a \geq \frac{M L_1 }{1-b}\left( \frac{\bar{L}_{1}}{(1-b)^{\tau}} + \bar{L}_2 \right).
\end{equation} 
Nevertheless, condition~\eqref{condition} is less conservative than~\eqref{conditionL}, since the latter requires a larger lower bound of the parameter $a$ in the event-triggering condition.
\end{remark}

{We have successfully extended the idea in~\cite{KZ-BG-EB:2022} for continuous-time systems to deal with event-triggered stabilization of discrete-time delay systems. It should be noted that stability of the closed-loop systems is ensured in our result while the the guarantee of stability is missing for the continuous-time systems in~\cite{KZ-BG-EB:2022}. Moreover, the discrete-time control systems do not exhibit Zeno behavior that is a phenomenon exists when infinite many control updates are triggered over a finite time interval. Excluding Zeno behavior is a main difficulty in the design of event-triggered control algorithms for continuous-time systems, while one of the main challenges in this study of discrete-time control systems is to rule out the trivial control updates.} Theorem~\ref{Th1} shows that the execution rule~\eqref{event} guarantees the inter-execution times are bounded from below by two time steps, provided $b<c$ and inequality~\eqref{condition} holds. {Hence, the existence of  strongly nontrivial sequence of event times is guaranteed.}  Under Assumption~\ref{Assumption2}, the Lipschitz constants involved in~\eqref{condition} rely on the tunable parameters $\sigma$, $a$, and $b$. Selecting appropriate $\sigma$, $a$, and $b$ according to~\eqref{condition} could be complicated for general nonlinear systems with time delays. Nevertheless, an easily verifiable condition can be obtained if all the functions mentioned in Assumption~\ref{Assumption2} satisfy global Lipschitz condition (that is, the radius $R$ in Definition~\ref{Lipschitz} is unbounded), and then we can provide a step-by-step guide to tune the parameters $\sigma$, $a$, and $b$ so that~\eqref{condition} is satisfied. More details are provided in the next corollary with its proof.

\begin{corollary}\label{corollary0} Suppose that all the Lipschitz conditions in Assumption~\ref{Assumption2} hold globally, and the following inequality regarding the Lipschitz constants is satisfied
\begin{equation}\label{condition2}
\mu>L L_1 \left( \bar{L}_{1,1} +\bar{L}_{1,2}+ \bar{L}_2 \right),
\end{equation}
then, there exist constants $\sigma\geq 0$, $a>0$, and $b>0$, so that both $\mu-\sigma>b$ and~\eqref{condition} are satisfied.
\end{corollary}
\begin{proof}
We conclude from~\eqref{condition2} that there exists a nonnegative constant $\sigma$ close to zero such that
\begin{equation}\label{condition3}
c:=\mu - \sigma>L L_1 \left( \bar{L}_{1,1} +\bar{L}_{1,2}+ \bar{L}_2 \right).
\end{equation}
Since Assumption~\ref{Assumption2} holds globally, all the Lipschitz constants in~\eqref{condition2} are independent of the choices of $a$, $b$, and the initial function $\varphi$. Then, for a given $\varphi\in\mathcal{C}_{\tau}$, we can find a big enough constant $a>0$ so that
\begin{equation}\label{condition4}
1> L_1 \left( \bar{L}_{1,1} +\bar{L}_{1,2}+ \bar{L}_2 \right) \left( \frac{L}{c} +\frac{\tilde{M}}{a} \right).
\end{equation}
Finally, we can identify a positive constant $b$ smaller than $c$ and close enough to zero  such that
\begin{equation}\label{condition5}
1\geq \frac{L_1}{1-b} \left( \bar{L}_{1,1} +\frac{\bar{L}_{1,2}}{(1-b)^{\tau}}+ \bar{L}_2 \right) \left( \frac{L}{c-b} +\frac{\tilde{M}}{a} \right),
\end{equation}
which, by multiplying both sides with $a$, is equivalent to~\eqref{condition}.
\end{proof}

%\margin{Is there a way to jointly optimize these choices, basically by taking $ a $ and $ b $ such that $ 1 $ is greater than or equal to their maximum?}

If all the conditions of Corollary~\ref{corollary0} are satisfied, and Assumption~\ref{Assumption} holds with the functions described in Theorem~\ref{Th1}, then suitable parameters $\sigma$, $a$, and $b$ in the execution rule~\eqref{event} can be derived by solving inequalities~\eqref{condition3},~\eqref{condition4}, and~\eqref{condition5}, orderly and respectively (see Section~\ref{Sec05} for examples with numerical simulations).

\section{The Linear Case}\label{Sec04}

In this section, we apply our results to the following linear time-delay system
\begin{eqnarray}\label{lsys}
\left\{\begin{array}{ll}
x(k+1)= A_1 x(k) + A_2 x(k-\tau) + B u(k)\cr
x_0=\varphi
\end{array}\right.,
\end{eqnarray}
where state $x(k)\in\mathbb{R}^n$, control input $u(k)=K x(k)\in\mathbb{R}^m$ for some $n,m\in \mathbb{N}$, initial function $\varphi\in\mathcal{C}_{\tau}$, and time delay $\tau=1$. Matrices $A_1$, $A_2$, $B$, and $K$ are with appropriate dimensions. The sampled-data implementation of the feedback control is $u(k)=K x(k_i)$, for $k\in[k_i,k_{i+1})$, and system~\eqref{lsys} can be written in the form of~\eqref{CL.sys} with measurement error $e$. The sequence of event times $\{k_i\}_{i\in\mathbb{N}}$ is to be determined by~\eqref{et.time}.

For system~\eqref{lsys}, we have $p(x)=K x$ which is globally Lipschitz and
\begin{align*}
\bar{f}(\phi,u)=&\phi(0)-f(\phi,u) \cr
               =&\phi(0)- \left( A_1 \phi(0) + A_2 \phi(-1) + B u \right) \cr
               =&(I-A_1)\phi(0)- A_2 \phi(-1) - B u.
\end{align*}
{It is easy to see that the function $(\phi,x)\mapsto \bar{f}(\phi,Kx)$  is globally Lipschitz with Lipschitz constants $\bar{L}_{1,1}=\|I-A_1\|$, $\bar{L}_{1,2}=\|A_2\|$, and $\bar{L}_2=\|BK\|$. }

To derive the execution rule~\eqref{event} for system~\eqref{lsys}, we consider the Lyapunov candidate with $V_1(\phi(0))=\|\phi(0)\|$ and $V_2(\phi_{\tau\setminus 0})=\varepsilon \|\phi(-1)\|$, where
\[
\varepsilon=\frac{1}{2}\left( -\|A_1+BK\| +\sqrt{\|A_1+BK\|^2+4\|A_2\|} \right).
\]
Then $V(\phi)=V_1(\phi(0))+V_2(\phi_{\tau\setminus 0})=\|\phi(0)\| + \varepsilon \|\phi(-1)\|$, and conditions (i) and (ii) of Theorem~\ref{Th.ISS} are satisfied with 
\[
\alpha_1(r)=\alpha_2(r)=r \textrm{ and } \alpha_3(r)=\varepsilon r,
\]
for $r\geq 0$. It can be seen that $\alpha^{-1}_1(s)=s$ for $s\geq 0$ is globally Lipschitz on its domain with Lipschitz constant $L_1=1$.

Next, we will check condition (iii) for system~\eqref{lsys} with measurement error $e$. From the discrete-time dynamics of system~\eqref{lsys}, we have
\begin{align*}
&V(x_{k+1})-V(x_k)  \cr
  =  &  \|x(k+1)\|+\varepsilon\|x(k)\|-  \|x(k)\|-\varepsilon\|x(k-1)\| \cr
\leq & \left( \|A_1+BK\| +\varepsilon-1 \right)\|x(k)\| + (\|A_2\|-\varepsilon) \|x(k-1)\| \cr
     & +\|BK\|~\|e(k)\| \cr
  =  & -(1-\varepsilon-\|A_1+BK\|)  V_1(x(k)) - \left(1-\frac{\|A_2\|}{\varepsilon}\right)V_2((x_k)_{\tau\setminus 0}) \cr
     & +\|BK\|~\|e(k)\| \cr
\leq & -\mu V(x_k) +\chi(\|e(k)\|),
\end{align*}
where $\chi(r)=r \|BK\|$ for $r\geq 0$ and
\begin{align*}
\mu=&\min\left\{ 1-\varepsilon-\|A_1+BK\|,  1-\frac{\|A_2\|}{\varepsilon} \right\}\cr
   =& \frac{1}{2}\left(  2- \|A_1+BK\| -\sqrt{\|A_1+BK\|^2 + 4\|A_2\|} \right).
\end{align*}
We can see that the $\mathcal{K}$ class function $\chi$ is globally Lipschitz with Lipschitz constant $L=\|BK\|$.

Up to now, we have shown that Assumption~\ref{Assumption} is true if $\mu>0$, and all the Lipschitz conditions in Assumption~\ref{Assumption2} hold globally. If we require
\begin{align}\label{inequalityl}
& 2 -\|A_1+BK\| - \sqrt{\|A_1+BK\|^2+4\|A_2\|}  \cr
\geq ~& 2\|BK\| (\|I-A_1\| + \|A_2\| + \|BK\|),
\end{align}
then $\mu>0$ and inequality~\eqref{condition2} hold, and we can obtain from Corollary~\ref{corollary0} that there exist constants~$\sigma\geq 0$,~$a>0$, and~$b>0$ so that $\mu-\sigma>b$ and the following inequality is satisfied
\begin{align}\label{conditionl}
1-b\geq \left( \|I-A_1\| +\frac{\|A_2\|}{1-b}+ \|BK\| \right) \left( \frac{\|BK\|}{\mu-\sigma-b} +\frac{\tilde{M}}{a} \right),
\end{align}
which is identical to~\eqref{condition}. Hence, we conclude from the above discussions with Corollary~\ref{corollary0} and Theorem~\ref{Th1} the following result for system~\eqref{lsys}.

\begin{corollary}\label{corollary}
If inequality~\eqref{inequalityl} holds for system~\eqref{lsys}, there exist constants~$\sigma\geq 0$,~$a>0$, and~$b>0$ such that $\mu-\sigma>b$ and~\eqref{conditionl} are both satisfied. Then, the sampled-data implementation of system~\eqref{lsys} with the event times $\{k_i\}_{i\in\mathbb{N}}$ determined by
\begin{equation}\label{et.timel}
k_{i+1}=\min\left\{ k> k_i \mathrel{\Big|} \|e(k)\| > \frac{\sigma}{\|BK\|} \|x(k)\| +  a(1-b)^k  \right\}
\end{equation}
is {GAS}. Moreover, the sequence of event times $\{k_i\}_{i\in\mathbb{N}}$ is strongly nontrivial.
\end{corollary}

\section{Examples}\label{Sec05}

In this section, two examples are investigated to illustrate the effectiveness of the obtained results with our event-triggering scheme. In the first example, we apply our results to a linear time-delay system.

\begin{example}\label{ex1}
Consider control system~\eqref{lsys} with
\[
A_1=\begin{bmatrix}
   0.95  &  0 \\
   0.01 &  1.05
\end{bmatrix}, ~A_2=\begin{bmatrix}
    0  &  -0.01 \\
    -0.01  &  0
\end{bmatrix}, 
\]
\[
B=\begin{bmatrix}
    3  &  0.2 \\
    0.5  &  1
\end{bmatrix}, ~K=\begin{bmatrix}
     -0.1621  &  ~~0.0324 \\
    ~~0.0810  &  -0.4862
\end{bmatrix}.
\]
\end{example}
Based on the discussion in Section~\ref{Sec04}, we have $\mu\approx 0.4030$ and~\eqref{inequalityl} holds. Since $\mu<1$, we can conclude from Theorem~\ref{Th.ISS} that feedback control system~\eqref{lsys} with the above given parameters is GAS. To derive the event times $\{k_i\}_{i\in\mathbb{N}}$ from~\eqref{et.timel}, the parameters $\sigma$, $a$, and $b$ in~\eqref{et.timel} can be chosen so that inequality~\eqref{conditionl} holds. Following the proving process of Corollary~\ref{corollary0}, all the parameters considered in this example are selected with initial condition $\varphi(s)=[1 ~1]^T$ for $s\in\mathbb{N}_{\tau}$ so that~\eqref{conditionl} is satisfied. 

For different combinations of $a$ and $b$, Table \ref{table} indicates the number of event times on the time interval $[0,10^4]$ when $\sigma=0.1$. We can see that increasing $b$ with fixed $a$ leads to a substantial increase of the event times on $[0,10^4]$, while increasing $a$ with unchanged $b$ slightly reduces the number of control updates over this finite time period. The reason that changing the value of $b$ affects more on the amount of event times is that the convergence speed of the Lyapunov candidate is closely related to $b$ (see~\eqref{Lyapunov.convergence} with the facts $b<c$ and $b=\eta$).

Table~\ref{table2} shows a performance comparison on the amount of event times over the time period $[0,10^5]$ between the execution rule~\eqref{event} and the time-dependent rule~\eqref{event2}. It can be observed that much less control updates are triggered by~\eqref{event}, which depends on both the system states and the evolution time. We further demonstrate the comparison in Fig.~\ref{fig1}. It can be observed that the error norm stays underneath the corresponding threshold in each subfigure of Fig.~\ref{fig1}. We can also see that more events are triggered in Fig.~\ref{fig1b}. The reason is that the execution rule~\eqref{event2} allows shorter time for the measurement error to evolve from zero to the time-dependent threshold. Corollary~\ref{corollary} states that the sequence of event times is strongly nontrivial, and this can be verified by Fig.~\ref{fig1} in which all the inter-event times in interval $[0,200]$ are larger than one.

%We choose $\sigma=0.1$, $a=3.5\|\varphi\|_{\tau}$, and $b=0.05$ so that inequality~\eqref{inequalityl} is satisfied. 

\vskip2mm

In the next example, we investigate a scalar control system to verify the effectiveness of our results on nonlinear time-delay systems.

\begin{table}[!t]
\caption{Number of the Event Times Determined by~\eqref{et.timel}  on the Time Interval $[0,10^4]$ with $\sigma=0.1$. {The first column under the category `Number of event times' is for system~\eqref{ex1} with initial condition $\varphi(s)=[1 ~1]^T$ for $s\in\mathbb{N}_{\tau}$, while the second column is derived with initial condition $\varphi(s)=[-2 ~3]^T$ for $s\in\mathbb{N}_{\tau}$}. }
\label{table}
\centering
\scalebox{1.}{
\begin{tabular}{ |c|c|c|c| } 
\hline
$a$ & $b$ & \multicolumn{2}{c|}{Number of event times} \\ 
\hline
$16$ & 0.01 & 2135 & {2141}\\ \cline{1-4}
$16$ & 0.03 & 2317 & {2323}\\ \cline{1-4} 
$24$ & 0.03 & 2315 & {2320}\\ \cline{1-4}
\hline
\end{tabular} }
%\caption{Table to test captions and labels}
%\label{table:1}
\end{table}

\begin{table}[!t]
\caption{Number of the Event Times Determined by~\eqref{et.timel} on the Time Interval $[0,10^5]$. {The first column regarding to the category of `Number of event times' is for system~\eqref{ex1} with initial condition $\varphi(s)=[1 ~1]^T$ for $s\in\mathbb{N}_{\tau}$, and the second column is derived with initial condition $\varphi(s)=[-2 ~3]^T$ for $s\in\mathbb{N}_{\tau}$}.}
\label{table2}
\centering
\scalebox{1.}{
\begin{tabular}{ |c|c|c| } 
\hline
$a$ & \multicolumn{2}{|c|}{$16$}  \\ \cline{1-3}
$b$ & \multicolumn{2}{|c|}{$0.01$}  \\ \cline{1-3}
Number of event times with $\sigma=0.1$ & 15845 & {15857 } \\ \cline{1-3}
Number of event times with $\sigma=0$~~   & 17373 & {17369 }\\ \cline{1-3}
\hline
\end{tabular} }
%\caption{Table to test captions and labels}
%\label{table:1}
\end{table}

\begin{figure}[!t]
\centering
\subfigure[$\sigma=0.1$]{\label{fig1a}\includegraphics[width=3.45in]{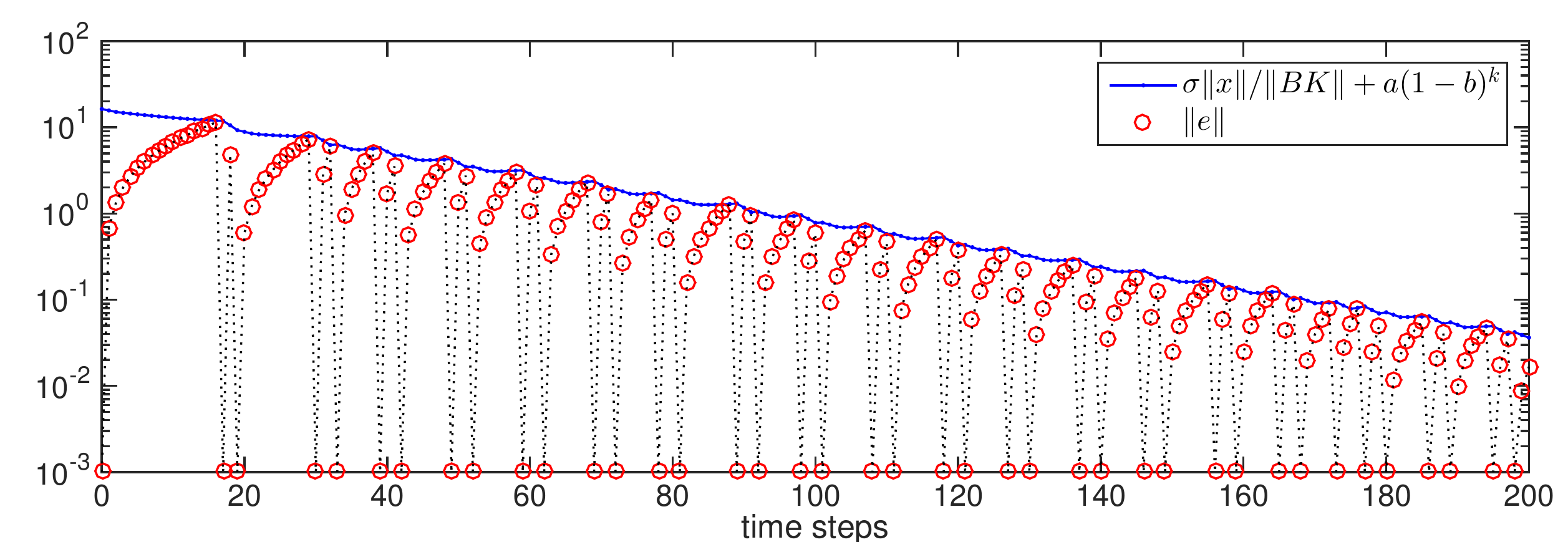}}
\subfigure[$\sigma=0$]{\label{fig1b}\includegraphics[width=3.45in]{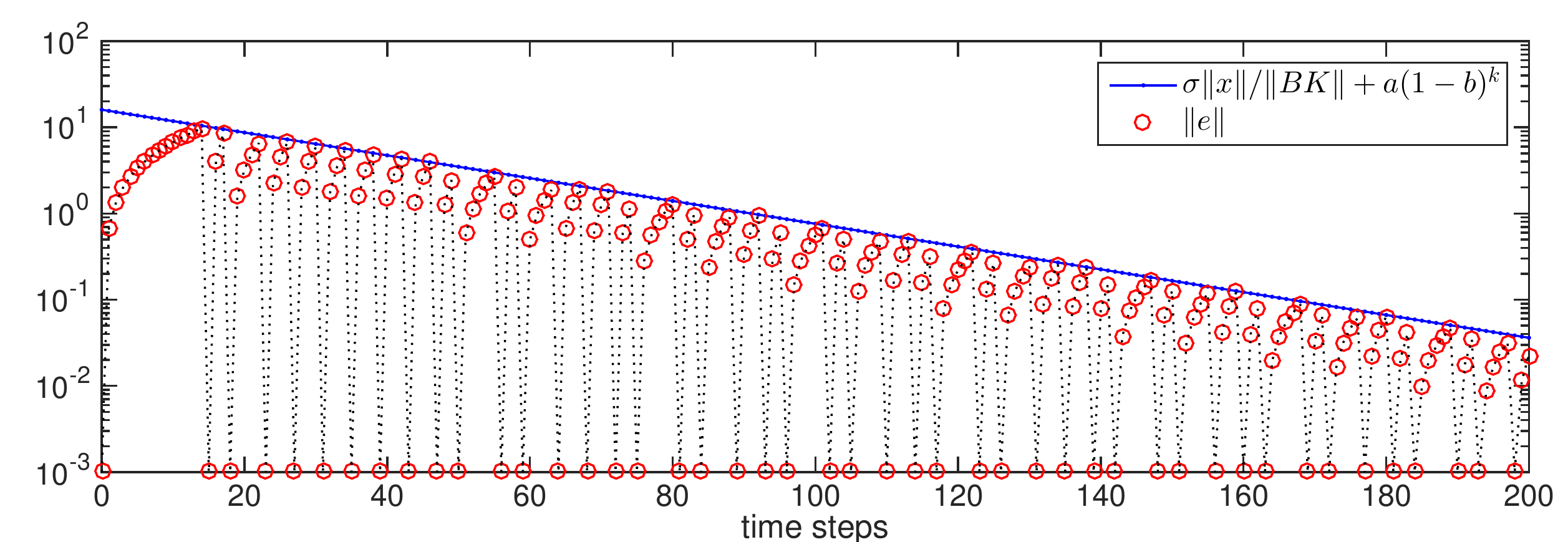}}
\caption{Simulations of Example~\ref{ex1} with $a=16$ and $b=0.03$. The red stems denote the evolution of the error norm $\|e\|$ which is traced by the black dotted line. The red stems on the time axis also indicate the event times. The blue lines represent the dynamic threshold $\sigma\|x\|/\|BK\| + a(1-b)^k$.}
\label{fig1}
\end{figure}

\begin{example}\label{ex2}
Consider the following discrete-time control system with time delay 
\begin{eqnarray}\label{sys.ex2}
\hskip-7mm
\left\{\begin{array}{ll}
x(k+1)= A_1 x(k) + A_2 \cos(x(k)) \sin(x(k-\tau)) + B u(k)\cr
x_0=\varphi \cr
\end{array}\right.,
\end{eqnarray}
where $x(k)\in\mathbb{R}$, $A_1=1$, $A_2=0.05$, $B=2$, time delay $\tau=1$, control input $u(k)=Kx(k)$ with control gain $K=-0.3$, and initial condition $\varphi(s)=0.2$ for $s\in [-\tau,0]$.
\end{example}

Similarly to the analysis of system~\eqref{lsys}, we can derive from the dynamics of control system~\eqref{sys.ex2} the Lipschitz conditions $\bar{L}_{1,1}=|1-A_1|$, $\bar{L}_{1,2}=|A_2|$, and $\bar{L}_2=|BK|$. Consider the Lyapunov candidate $V(\phi)=V_1(\phi(0))+V_2(\phi_{\tau\setminus 0})$ with $V_1(\phi(0))=|\phi(0)|$ and $V_2(\phi_{\tau \setminus 0})=\epsilon|\phi(-1)|$, where constant $\epsilon=0.1$, and then conditions (i) and (ii) of Theorem~\ref{Th.ISS} hold with $\alpha_1(|\phi(0)|)=\alpha_2(|\phi(0)|)=|\phi(0)|$ with Lipschitz constant $L=1$ and $\alpha_3(|\phi|_{\tau\setminus 0})=\epsilon|\phi|_{\tau\setminus 0}$. Following the discrete dynamics of system~\eqref{sys.ex2} with the event-triggered implementation, we can show that condition (iii) of Theorem~\ref{Th.ISS} is satisfied with $\mu=\min\{1-\epsilon-|A_1+BK|, 1-|A_2|/\epsilon\}$ and $\chi(|e|)=|BK||e|$ which has Lipschitz constant $L=|BK|$. In the simulation, we select $a=2.2$, $b=0.02$, and $\sigma=0.05$ so that both $b<\mu-\sigma$ and~\eqref{condition} are satisfied. Hence, Theorem~\ref{Th1} concludes that the event-triggered implementation of control system~\eqref{sys.ex2} with triggering condition~\eqref{event} is GAS, and the sequence of event times is strongly nontrivial.

Fig.~\ref{fig2} shows evolution of the measurement error and the dynamic thresholds with the event times determined by~\eqref{event} and the state-dependent execution rule~\eqref{event1}, respectively. In Fig.~\ref{fig2a}, we can see that the error norm never surpasses the event-triggering threshold $\sigma|x|/|BK|+a(1-b)^k$, and all the inter-execution times over the time interval $[0,200]$ are at least two, which is in accordance with our theoretical results. It can be observed from Fig.~\ref{fig2b} that the error $e$ stays zero for all $k\in \mathbb{Z}^+$, that is, the control input $u$ is updated at every time step. Therefore, the state-dependent execution rule~\eqref{event1} reduces the event-triggering implementation to the traditional feedback control mechanism. {Fig.~\ref{fig2cd} shows the trajectories of system state and control input for system~\eqref{ex2} with the event times determined by~\eqref{event}. It can be observed that control updates are permitted only when the events are detected. This is the major difference from time-triggered control paradigm, such as periodic or aperiodic sampled-data control, which requires the control updates to be executed according to a predetermined schedule (see, e.g, \cite{RP-DN:2016}).}

\begin{figure}[!t]
\centering
\subfigure[$a=2.2$]{\label{fig2a}\includegraphics[width=3.45in]{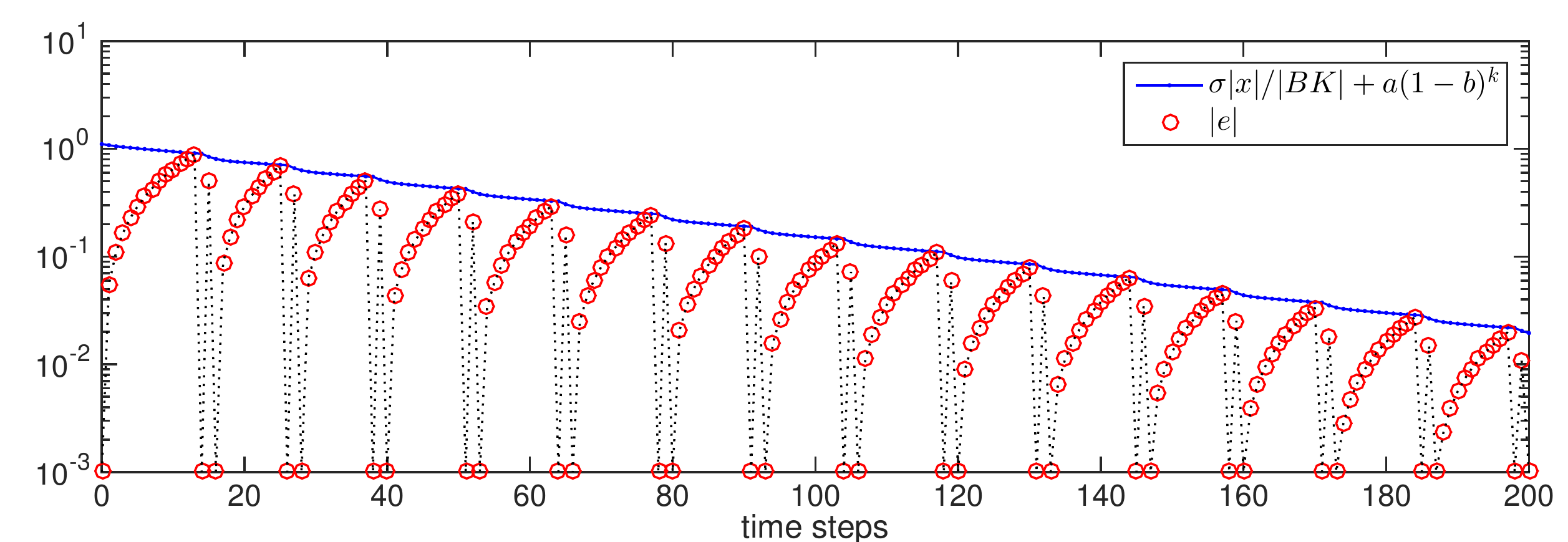}}
\subfigure[$a=0$]{\label{fig2b}\includegraphics[width=3.45in]{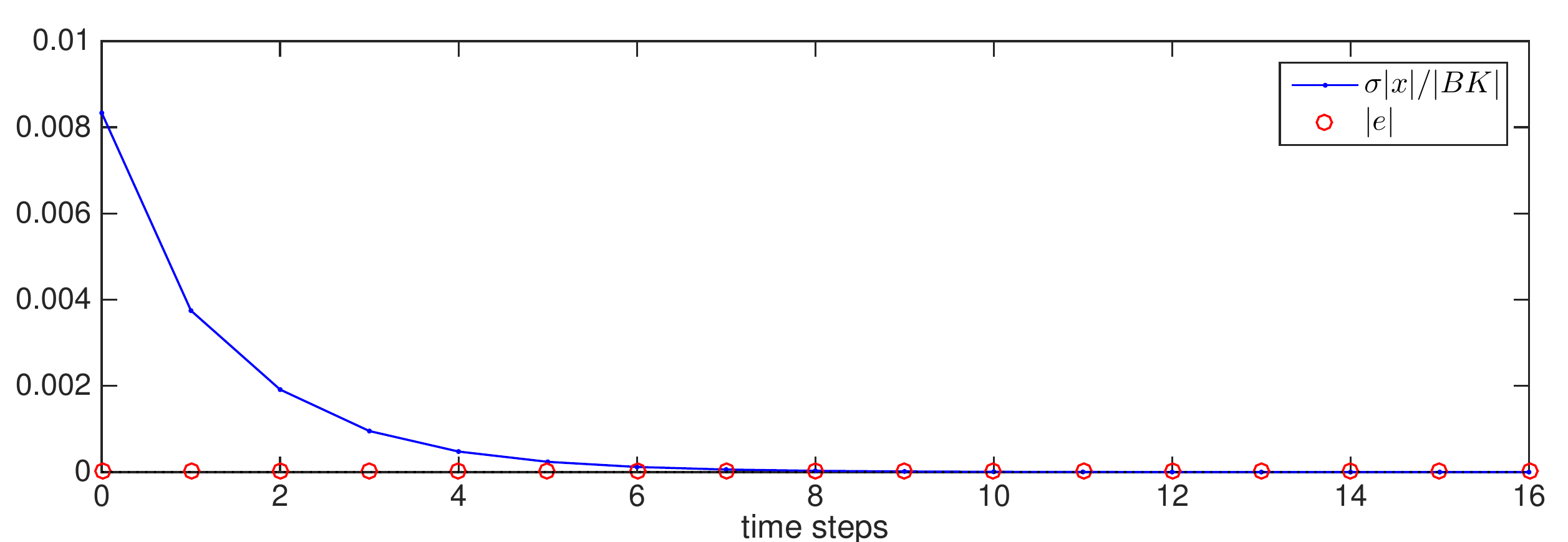}}
\caption{ Simulations of Example~\ref{ex2} with $\sigma=0.05$, $b=0.02$, and initial function $\varphi(s)=0.2$ for $s\in\mathbb{N}_{-\tau}$. The evolution of the error norm $\|e\|$ is traced by the black dotted line.}
\label{fig2}
\end{figure}

\begin{figure}[!t]
\centering
\subfigure[System state $x$]{\label{fig2c}\includegraphics[width=3.45in]{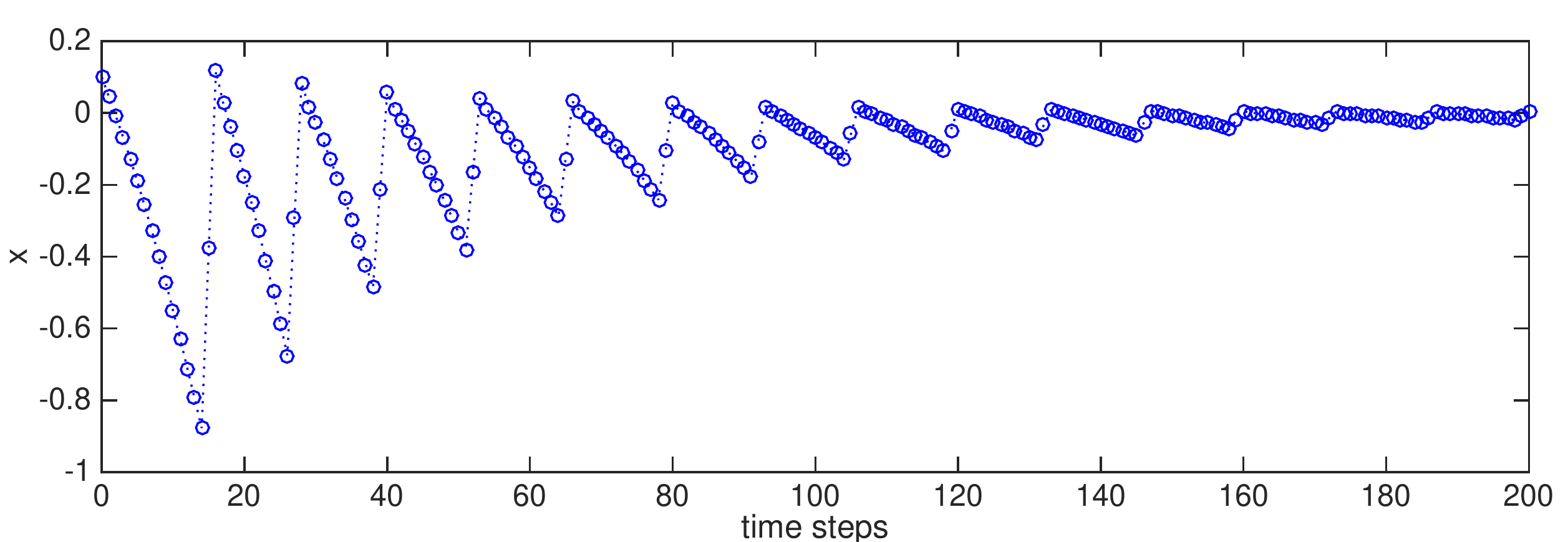}}
\subfigure[Control input $u$]{\label{fig2d}\includegraphics[width=3.45in]{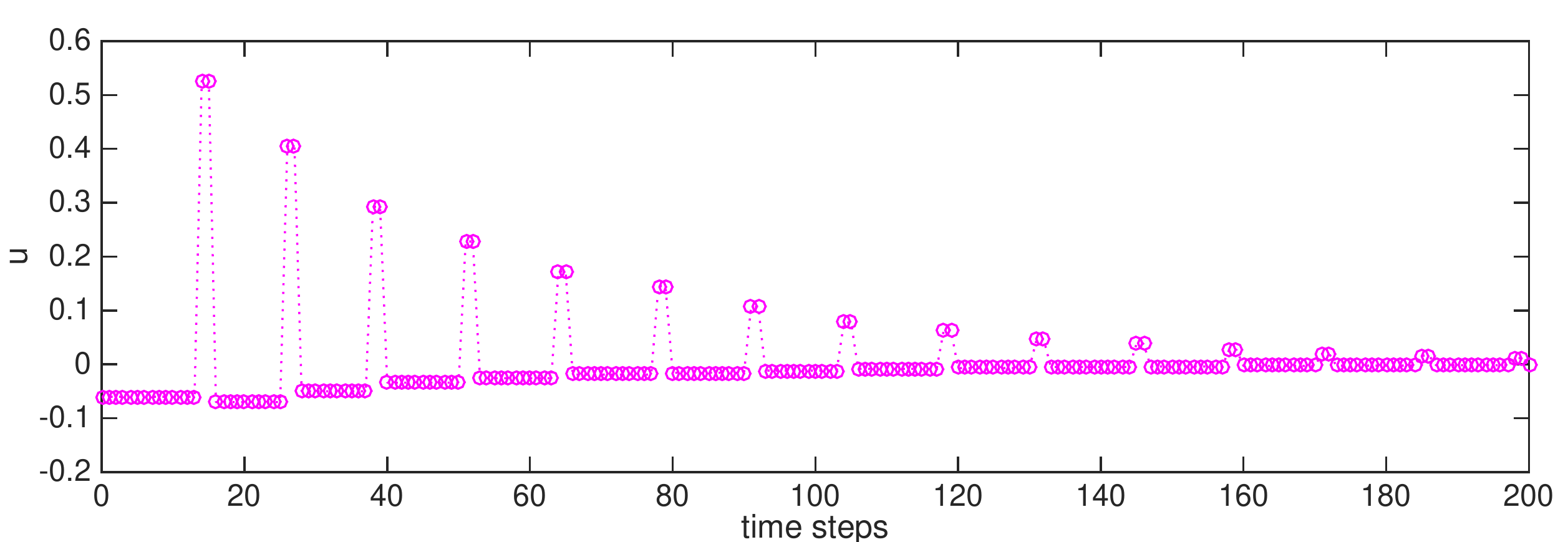}}
\caption{Simulations of the system state and control input for Example~\ref{ex2} with $\sigma=0.05$, $a=2.2$, $b=0.02$, and initial function $\varphi(s)=0.2$ with $s\in\mathbb{N}_{-\tau}$. These circles represent the system states (or control inputs) which are traced by dotted lines.}
\label{fig2cd}
\end{figure}

\section{Conclusions}\label{Sec06}

In this paper, we have introduced an event-triggering scheme to update the control inputs for discrete-time systems with time delays. Sufficient conditions have been derived to guarantee the event-triggered control systems are {globally asymptotically stable}. Moreover, the lower bound of the inter-execution times has been proved to be bigger than one, which excludes the traditional feedback control of updating the input signals at every time step from our event-triggering scheme. 

{Avenues of future work include:
\begin{itemize}
\item Investigating the implications of the proposed algorithm in consensus and synchronization problems over networks;
\item {Improving the existing ISS results for linear time-delay systems and deriving the event-triggering algorithm accordingly;}
\item {Exploring new event-triggering algorithms} along the line of the study in~\cite{RP-PT-DN-AA:2014,WW-DN-RP-WH:2020} by converting the time-dependent portion of the event-triggering threshold in~\eqref{event} to a threshold variable regulated by a difference equation;
\item Considering time-delay effects not only in the system dynamics but also in the event-triggered feedback controllers.
\end{itemize}
}

%\end{proof}

%% References


%\bibliographystyle{elsarticle-num}



\end{document}